\documentclass[11pt]{elsarticle}

\usepackage{datetime}
\settimeformat{ampmtime}
\usdate

\newcommand{\nwc}{\newcommand}
\nwc{\draftdate}{\today}
\usepackage{amsmath,amssymb}
\usepackage{amsthm}

\allowdisplaybreaks

\newtheorem{thm}{Theorem}[section]
\newtheorem{lemma}[thm]{Lemma}

\newcommand{\barint}{\hbox{$\int$\kern-0.75\intwidth
\vrule width 0.5\intwidth height 2.4pt depth -2pt\kern0.25\intwidth}}
\newlength\intwidth
\setbox0=\hbox{$\int$}
\intwidth=\wd0

\newcommand\avint{\hbox{\hbox{$\displaystyle \int$}\hbox{\kern-.9em{$-$}}}}

\newcommand\smavint{\hbox{\hbox{$\int$}\hbox{\kern-.75em{$-$}}}}

\nwc{\st}{^{\mbox{\it st}}}

\nwc{\qref}[1]{(\ref{#1})}
\nwc{\veloc}{v}
\nwc{\rhoc}{\beta}
\nwc{\hl}{\hat{L}}

\def\Xint#1{\mathchoice
{\XXint\displaystyle\textstyle{#1}}%
{\XXint\textstyle\scriptstyle{#1}}%
{\XXint\scriptstyle\scriptscriptstyle{#1}}%
{\XXint\scriptscriptstyle\scriptscriptstyle{#1}}%
\!\int}
\def\XXint#1#2#3{{\setbox0=\hbox{$#1{#2#3}{\int}$}
\vcenter{\hbox{$#2#3$}}\kern-.51\wd0}}

\def\dashint{\Xint-}

\nwc{\intRp}{\int_0^\infty}
\nwc{\aint}{\dashint}
\nwc{\aaint}{\dashint}

\newcommand{\cF}{{\cal F}}

\newcommand{\R}{\mathbb R}

\newcommand{\be}{\begin{eqnarray}}
\newcommand{\ee}{\end{eqnarray}}
\newcommand{\nn}{\nonumber}
\newcommand{\ben}{\begin{eqnarray*}}
\newcommand{\een}{\end{eqnarray*}}

\journal{arXiv.org}

\begin{document}
	
\title{On nonnegative solutions for the Functionalized Cahn-Hilliard equation with degenerate mobility}
\author[1]{Shibin Dai\corref{cor1}}
\ead{sdai4@ua.edu}
\author[2]{Qiang Liu} 
\ead{matliu@szu.edu.cn}

\author[1]{Toai Luong}
\ead{ttluong1@crimson.ua.edu}

\author[3]{Keith Promislow}
\ead{PROMISLO@msu.edu}
\cortext[cor1]{Corresponding author.}

\address[1]{Department of Mathematics, University of Alabama, Tuscaloosa, AL 35487-0350, USA. 
} 
\address[2]{
College of Mathematics and Statistics, Shenzhen University, Shenzhen, 518060, China. 
} 
\address[3]{
Department of Mathematics, Michigan State University, East Lansing, MI 48824, USA.
}

\begin{abstract}
The Functionalized Cahn-Hilliard equation has been proposed as a model for the interfacial energy of 
phase-separated mixtures of amphiphilic molecules. 
We study the existence of a nonnegative weak solutions of a gradient flow of the Functionalized 
Cahn-Hilliard equation subject to a degenerate mobility $M(u)$ 
that is zero for $u\leq 0$. Assuming the initial data $u_0(x)$ is positive,  we construct a weak solution as 
the limit of solutions corresponding to non-degenerate mobilities and verify that it satisfies an energy 
dissipation inequality. 

\end{abstract}


\maketitle

\section{Introduction}
The Functionalized Cahn-Hilliard (FCH) free energy was introduced in \cite{promislow-2011}. It is 
an extension of the model of Gompper and Goos \cite{GG-94},
proposed to describe the  free energy of microemulsions of amphiphilic molecules and solvent. 
Amphiphilic molecules are formed by chemically bonding two components whose individual interactions 
with the solvent are energetically favorable and unfavorable, respectively. 
When blended with the solvent, amphiphilic molecules have a propensity to phase separate, 
forming amphiphilic rich domains that are thin, generically the thickness of two molecules, in at least 
one direction. For a binary mixture with composition described by $u$ on $\Omega\subset\R^d$, the 
FCH free energy takes the form
\begin{align}\label{fch-ener}
	\cF(u)=\int_\Omega\frac{1}{2}|-\Delta u+W'(u)|^2-\eta\left(\frac{1}{2}|\nabla u|^2+W(u)\right)dx,
\end{align}
where $\Omega$ is a bounded domain in $\R^d$ with the boundary $\partial\Omega$. 
The FCH equation, which is the gradient flow of the FCH energy functional, is written as 
\begin{align}
	\label{fch-eqn1} u_t&=\nabla\cdot(M(u)\nabla\mu), \quad (x,t)\in\Omega_T:=\Omega\times(0,T), \\	
	\label{fch-eqn2}\mu&=-\Delta\omega+W''(u)\omega-\eta\omega, \\
	\label{fch-eqn3}\omega&=-\Delta u+W'(u),
\end{align}	
where $T>0$ is a given number. 
This equation is often subject to periodic or zero-flux boundary conditions on $\partial\Omega$. 
We also prescribe initial values $u(x,0)=u_0(x)$ for all $x\in\Omega$, where $u_0\in H^2(\Omega)$ is 
a given function. 
The function $\mu$ is the chemical potential defined by the first variational derivative of the FCH 
energy functional \qref{fch-ener}. 
The diffusion mobility $M: \R\to [0,\infty)$ is nonnegative and continuous. 
The double well potential $W: \R\to\R$ is smooth enough and has two unequal depth local minima 
at 0 and $b_+>0$ for which $W(0)=0>W(b_+)$, and $W'$ has exactly three zeroes at 0, $b_+$, and 
$b_0 \in (0,b_+)$. 
The parameter $\eta>0$ characterizes key structural properties of the amphiphilic molecules. 
The function $u$ is the order parameter, representing the relative volume fraction of amphiphilic 
materials, with $u\equiv 0$ being pure solvent and $u\equiv b_+$ being pure amphiphile. 

Highly amphiphilic lipids have long hydrophobic tails.  The energy of low concentrations of highly 
amphiphilic materials grows exponentially with the tail length \cite{Bates:highlyamphiphilic}. This leads 
to models in which the concavity of the left well, $W''(0)$, is large \cite{FCH:benchmark}.  Indeed, 
arguing formally,
a perturbation $v$ of a $u\equiv 0$ distribution satisfies the linear diffusion equation
\begin{align} \label{FCH-zero}
u_t = \Delta M(0)\left( -\Delta W''(0)-\eta\right)\left(-\Delta + W''(0)\right) u.
\end{align}
To prevent spuriously high diffusivity at low concentrations, it  is natural to take the mobility of 
lipids so that the product $M(0)(W''(0))^2$ remains bounded. 
To compensate for the high energy of dispersed amphiphilic molecules  requires mobilities that are zero  
or asymptotically zero. In this paper we fix the low-density energy and establish the existence and 
energy dissipation of weak solutions with vanishing mobility.

In \cite{dlp:weakFCH} it was proved that for  
\begin{align} \label{M-abs}
	M(u)=|u|^m,
\end{align}
there exists a weak solution for this class of degenerate FCH equation. 
Here $0<m<\infty$ if the spatial dimension $d\leq 4$, and $0<m<4/(d-4)$ if $d\geq 5$.
The degenerate mobility \qref{M-abs} is an 
appropriate choice for the degenerate Cahn-Hilliard (CH) equation, which models phase separations in 
composite materials 
\cite{dd:onesidedCH, dd:degenerateCH, dd:numericCH,dd:weak, lms:degenerateCH}. 
 It has been shown by 
numerical simulations and by asymptotic analysis that for $d\geq 2$
the degenerate mobility \qref{M-abs} does not guarantee the 
weak solution for the degenerate CH to remain positive, even if the initial value is positive 
\cite{dd:onesidedCH, dd:degenerateCH, dd:numericCH, lms:degenerateCH}.
This is a consequence of the  Gibbs-Thomson effect 
(see also \cite{cen:CH, lms:insufficient, lms:response, voigt:comment} for discussions on the 
role of degenerate diffusions, and \cite{yin:1992} for the one-dimensional case).  
The situation is different for the degenerate FCH equation, since there is no Gibbs-Thomson effect
\cite{dp:bilayer,dp:competitive}. Instead, a  more important feature we need to guarantee is that the relative  
volume fraction of amphiphilic materials $u$ remains nonnegative. 
It is the purpose of this paper to show that, with a specific form of degenerate mobility 
\begin{align}\label{M(u)}
	M(u)=\left\{\begin{array}{l} u, \; u>0, \\
	0, \; u\leq 0, \end{array} \right.
\end{align}
 there exists
a nonnegative weak solution to the FCH equation \qref{fch-eqn1}-\qref{fch-eqn3} when the initial 
data $u_0(x)$ is positive for all $x\in\Omega$.
Since the FCH equation is a gradient flow for the FCH functional, it is natural to expect the weak solution 
to satisfy an energy dissipation inequality for $\cF(u)$. 
In \cite{dlp:weakFCH} this energy dissipation property for $\cF$ was stated informally and without a proof. 
In this paper we establish that the weak solutions dissipate the FCH energy. 

\subsection{Main result}
In this paper, we assume the dimension $d=1,2,3$ and the set $\Omega=(0,2\pi)^d$, 
and consider the periodic boundary condition on the boundary $\partial\Omega$. 
We choose the mobility $M(u)$ to be \qref{M(u)}, 
which is degenerate at $u=0$. 
The degeneracy of mobility at $u=0$ presents the technical difficulties. 
We also assume that $W\in C^4(\R)$ and there exist positive constants $C_1,C_2,C_3,C_4$ such that for all $z\in\R$,
\begin{align}
	\label{grow-1}C_1|z|^{2p}-C_2\leq&W(z)\leq C_3|z|^{2p}+C_4, \\
	\label{grow-2}&|W'(z)|\leq C_3|z|^{2p-1}+C_4, \\
	\label{grow-3}C_1|z|^{2p-2}-C_2\leq&W''(z)\leq C_3|z|^{2p-2}+C_4, \\
	\label{grow-4}&|W'''(z)|\leq C_3|z|^{2p-3}+C_4,\\
	\label{grow-5}2C_1|z|^{2p}-C_2\leq&W'(z)z,
\end{align} 
for a constant $1<p<\infty$ if $d=1,2$, and $1<p\leq (d-1)/(d-2)$ if $d=3$. 
One example of such a potential $W$ is
\begin{align*}
	W(u)=u^2(u-1)(u-2),
\end{align*}
with $p=2$. 
Under these assumptions, along with the cut-off degenerate mobility \qref{M(u)}, 
we will prove that the FCH equation \qref{fch-eqn1}-\qref{fch-eqn3} has a nonnegative weak solution that is not zero everywhere in $\Omega_T$, assuming that the initial data $u_0(x)$ is positive in $\Omega$.

Our analysis follows the same strategy as in \cite{dd:weak, dlp:weakFCH, Elliott-96} and involves two steps. 
The first step is to approximate the degenerate mobility $M(u)$ by a non-degenerate mobility $M_\theta(u)$ defined for 
$\theta\in(0,1)$ by
\begin{align}\label{M-theta}
	M_\theta(u)=\left\{\begin{array}{l} u, \; u>\theta, \\
	\theta, \; u\leq\theta . \end{array} \right.
\end{align}
The positive lower bound of $M_\theta(u)$ allows us to find a sufficiently regular weak solution to \qref{fch-eqn1}-\qref{fch-eqn3} with the positive mobility $M_\theta(u)$.

\begin{thm}\label{thm-main1}
	Let $u_0\in H^2(\Omega)$. With the potential $W(u)$ satisfying \qref{grow-1}-\qref{grow-5} and the mobility  
	$M_\theta(u)$ defined by \qref{M-theta}, for any given constant $T>0$, there exists a function $u_\theta$ that 
	satisfies the following conditions:
	\begin{enumerate}	
		\item $u_\theta\in L^\infty(0,T;H^5(\Omega))\cap C([0,T];H^l(\Omega))\cap C([0,T];C^{3,\alpha}(\bar{\Omega}))$, 
		where $1\leq l \leq 4$ and $0<\alpha<1/2$.
		
		\item  $\partial_tu_\theta\in L^2(0,T;(H^2(\Omega))')$.
		
		\item $u_\theta(x,0)=u_0(x)$ for all $x\in\Omega$.
		
		\item $u_\theta$ satisfies the FCH equation \qref{fch-eqn1}-\qref{fch-eqn3} in the following 
		weak sense:
		\begin{align}\label{FCH-w2}
		&\int_0^T\left\langle \partial_tu_\theta,\xi\right\rangle_{(H^2(\Omega)',H^2(\Omega))}dt 
		=-\int_0^T\int_\Omega M_\theta(u_\theta)\biggl(-\nabla\Delta\omega_\theta \nonumber \\
		& \qquad+W'''(u_\theta)\nabla u_\theta\omega_\theta+W''(u_\theta)\nabla\omega_\theta
		-\eta\nabla\omega_\theta\biggr)\cdot\nabla\xi dxdt
		\end{align}
		for all $\xi\in L^2(0,T;H^2(\Omega))$, where $\omega_\theta=-\Delta u_\theta+W'(u_\theta)$. 				
		In addition, for any $t\geq 0$, the following energy inequality holds 
		\begin{align} 
		&\int_\Omega\frac{1}{2}|\omega_\theta(x,t)|^2-\eta\left(\frac{1}{2}|\nabla u_\theta(x,t)|^2
		+W(u_\theta(x,t))\right)dx  \nonumber
		\\
		&+\int_0^t\int_\Omega M_\theta(u_\theta(x,\tau))|-\nabla\Delta\omega_\theta(x,\tau)
		+W'''(u_\theta(x,\tau))
		\nabla u_\theta(x,\tau)\omega_\theta(x,\tau)   \nonumber \\
		&+W''(u_\theta(x,\tau))\nabla\omega_\theta(x,\tau)-\eta\nabla\omega_\theta(x,\tau)|^2dxd\tau 
		\nonumber \\
		&\leq\int_\Omega\frac{1}{2}|-\Delta u_0+W'(u_0)|^2-\eta\left(\frac{1}{2}|\nabla u_0|^2+W(u_0)\right)dx.
		\label{ener-ineq1}
		\end{align}
		
		\item \label{th1-ineq}
		 If $u_0(x)>0$ for all $x\in\Omega$, then for any $0<\theta<1$,
		\begin{align}\label{bnd-neg1}
		\textup{ess sup}_{0\leq t \leq T}\int_\Omega\left|(u_\theta(x,t))_-+\theta\right|^2dx
		\leq C(\theta^2+\theta+\theta^{1/2}),
		\end{align}
		where $(u_\theta)_-=\min\{u_\theta,0\}$, 
		and $C$ is a generic positive constant that may depend on $d,T,\Omega,\eta,u_0$ 
		and $C_j(j=1,2,3,4)$ but not 
		on $\theta$.
	\end{enumerate}
\end{thm}

The estimate \qref{bnd-neg1} in part \ref{th1-ineq} of Theorem \ref{thm-main1} is essential. 
It is the key to prove the existence of a nonnegative weak solution to the equation 
\qref{fch-eqn1}-\qref{fch-eqn3} with the degenerate mobility \qref{M(u)}.

The second step is to consider the limit of $u_\theta$ as $\theta\rightarrow 0$. 
The limiting function $u$ of $u_\theta$ does exist and, in the weak sense, solves the FCH equation 
\qref{fch-eqn1}-\qref{fch-eqn3} with the mobility $M(u)$ defined by \qref{M(u)}. 
It can be interpreted that $u$ solves the FCH equation in the open set $U_T=U\times(0,T)\subset\Omega_T$, 
where $U$ is any open subset of $\Omega$ with $\nabla\Delta^2u\in L^q(U_T)$ for some $q>1$. 
As for the set where $u$ does not have enough regularity, that set is contained in the set where $M(u)$ is degenerate 
and another set of Lebesgue measure zero. 
Moreover, if the initial data $u_0(x)$ is positive in $\Omega$, we obtain a nonnegative weak solution to the equation 
\qref{fch-eqn1}-\qref{fch-eqn3} that is not constantly zero in $\Omega_T$.

\begin{thm}\label{thm-main2}
	Let $u_0\in H^2(\Omega)$. With the potential $W(u)$ satisfying \qref{grow-1}-\qref{grow-5} and 
	the mobility  $M(u)$ 
	defined by \qref{M(u)}, for any given constant $T>0$, there exists a function $u$ that satisfies the 
	following conditions:
	
	\begin{enumerate}
		\item $u\in L^\infty(0,T;H^2(\Omega))\cap C([0,T];H^1(\Omega))
		\cap C([0,T];C^\alpha(\bar{\Omega}))$ 
		where $0<\alpha<1/2$.
		
		\item $\partial_tu\in L^2(0,T;(H^2(\Omega))')$.

		\item $u(x,0)=u_0$.
		
		\item $u$ can be considered as a weak solution for the FCH 
		equation \qref{fch-eqn1}-\qref{fch-eqn3} in 
		the following weak sense:
		\begin{enumerate}
			\item Let $P$ be the set where $M(u)$ is not degenerate, that is,
			\begin{align} 
			P:=\{(x,t)\in\Omega_T:u(x,t)>0\}. \nn
			\end{align}
			There exists a set $B\subset\Omega_T$ with $|\Omega_T\backslash B|=0$ and a function 
			$\zeta:\Omega_T\rightarrow\R^d$ satisfying $\chi_{B\cap P}M(u)\zeta\in 
			L^2(0,T;L^{2d/(d+2)}(\Omega))$, 
			here $\chi_{B\cap P}$ is the characteristic function of $B\cap P$, such that
			\begin{align}\label{FCH-w3}
			\int_0^T\left\langle\partial_tu,\phi\right\rangle_{(H^2(\Omega)',H^2(\Omega))}dt
			=-\int_{B\cap P}M(u)\zeta\cdot\nabla\phi dxdt
			\end{align}
			for all $\phi\in L^2(0,T;H^2(\Omega))$.
			
			\item Let  $\nabla\Delta^2u$ be the generalized derivative of $u$ in the sense of distributions. 
			If for some open set $U\subset\Omega$, $\nabla\Delta^2u\in L^q(U_T)$ for some $q>1$, 
			where $U_T=U\times(0,T)$, then we have
			\begin{align*}
			\zeta=-\nabla\Delta\omega+W'''(u)\nabla u\omega+W''(u)\nabla\omega-\eta\nabla\omega  
			\quad in\;U_T,
			\end{align*}
			where $\omega=-\Delta u+W'(u)$.
			
			\item In addition, for any $t\geq 0$, the following energy inequality holds:
			\begin{align}\label{ener-ineq2}
			&\int_\Omega\frac{1}{2}|\omega(x,t)|^2-\eta\left(\frac{1}{2}|\nabla u(x,t)|^2+W(u(x,t))\right)dx 
			 \nonumber\\
			&+\int_{\Omega_t\cap B\cap P} M(u(x,\tau))|\zeta(x,\tau)|^2dxd\tau  \nonumber\\
			&\leq\int_\Omega\frac{1}{2}|-\Delta u_0+W'(u_0)|^2-\eta\left(\frac{1}{2}|\nabla u_0|^2
			+W(u_0)\right)dx.
			\end{align}		
		\end{enumerate}	
		
		\item If $u_0(x)>0$ for all $x\in\Omega$, then $u(x,t)\geq 0$ for all $(x,t)\in\Omega_T$, 
		and $u(x,t)$ is not constantly zero in $\Omega_T$.
	\end{enumerate}
\end{thm}

\subsection{Notation}
In this paper, we use $C$ to denote a generic positive constant that may depend on 
$d,T,\Omega,\eta,u_0$ and $C_j(j=1,2,3,4)$ but nothing else, in particular not on $\theta$. 
We also use $C_\theta$ to denote a generic positive constant that may depend 
on $d,T,\Omega,\eta,u_0, C_j(j=1,2,3,4)$ 
and $\theta$.

This paper is organized as follows. 
In Section \ref{pos-mobi} we prove Theorem \ref{thm-main1} using the Galerkin method. 
In Section \ref{dege-mobi} we prove Theorem \ref{thm-main2} that establishes the existence of a weak solution 
to the equation \qref{fch-eqn1}-\qref{fch-eqn3} with degenerate mobility.

\section{Weak solution for the positive mobility case} \label{pos-mobi}
In this section we prove Theorem \ref{thm-main1}. 
The proof for the existence of a weak solution $u_\theta$ can be found in Section 3 
in \cite{dlp:weakFCH}, which is 
based on Galerkin approximations. 
We just sketch the idea of that proof, as well as state the convergences and estimates that are 
necessary for later parts. 
The main purpose of this section is to prove the energy inequality \qref{ener-ineq1}, and the estimate 
\qref{bnd-neg1} 
when the initial data $u_0(x)$ is positive in $\Omega$.

\subsection{Galerkin approximation and Weak solution}
Let $\{\phi_j:j=1,2,...\}$ be the normalized eigenfunctions, in the sense that $\|\phi_j\|_{L^2(\Omega)}=1$, of the 
eigenvalue problem
\begin{align*}
	-\Delta u&=\lambda u \quad \mbox{in} \; \Omega
\end{align*}
subject to periodic boundary condition on $\partial\Omega$. 
The eigenfunctions $\phi_j$ are orthogonal in the $H^2(\Omega)$ and $L^2(\Omega)$ scalar product. Without loss of 
generality, we assume that $\lambda_1=0$, hence $\phi_1\equiv(2\pi)^{-d/2}$.

We consider the Galerkin approximation for the equation \qref{fch-eqn1}-\qref{fch-eqn3}:
\begin{align}
	u^N(x,t)&=\sum_{j=1}^{N}c^N_j(t)\phi_j(x),\quad \mu^N(x,t)=\sum_{j=1}^{N}d^N_j(t)\phi_j(x), \nn	\\	
	\int_\Omega\partial_tu^N\phi_jdx&=-\int_\Omega M_\theta(u^N)\nabla\mu^N\cdot\nabla\phi_jdx \label{ode1}\\	
	\int_\Omega\mu^N\phi_jdx&=\int_\Omega(-\omega^N\Delta\phi_j+W''(u^N)\omega^N\phi_j
	-\eta\omega^N\phi_j)dx \label{ode2}\\	
	u^N(x,0)&=\sum_{j=1}^{N}\left(\int_\Omega u_0\phi_jdx\right)\phi_j(x) \label{ode3}
\end{align}
where $\omega^N=-\Delta u^N+W'(u^N)$. 
This gives an initial value problem for a system of ordinary differential equations for $c^N_1,...,c^N_N$:
\begin{align}
	\partial_tc_j^N=&-\sum_{k=1}^{N}d_k^N\int_\Omega M_\theta\left(\sum_{i=1}^{N}c^N_i\phi_i\right) 
	\nabla\phi_k\cdot\nabla\phi_jdx, \label{ode4} \\	
	d_j^N=&(\lambda_j^2+\eta\lambda_j)c_j^N + \int_\Omega W''
	\left(\sum_{k=1}^{N}c^N_k\phi_k\right)W'\left(\sum_{k=1}^{N}c^N_k\phi_k\right)\phi_jdx   \nonumber \\
	&-(\lambda_j+\eta)\int_\Omega W'\left(\sum_{k=1}^{N}c^N_k\phi_k\right)\phi_jdx \nn\\
	&- \sum_{k=1}^{N}\lambda_kc_k^N\int_\Omega W''\left(\sum_{i=1}^{N}c^N_i\phi_i\right)\phi_k\phi_jdx, 
	\label{ode5}\\	
	c_j^N(0)=&\int_\Omega u_0\phi_jdx. \label{ode6}
\end{align}
Since the right hand side of \qref{ode4} depends continuously on $c^N_1,...,c^N_N$, the initial value 
problem \qref{ode4}-\qref{ode6} has a local solution. 
From the Subsections 2.2 and 3.1 in \cite{dlp:weakFCH}, we have the following estimates for $u^N$ and $\mu^N$.

\begin{lemma}\label{bnd-lem}
	Let $u^N$ be a solution of the system \qref{ode1}-\qref{ode3}, we have
	\begin{align}
		\label{bnd-uN-H2}\|u^N\|_{L^\infty(0,T;H^2(\Omega))}&\leq C, \\
		\label{bnd-M-muN}\int_0^T\int_\Omega M_\theta(u^N(x,\tau))|\nabla\mu^N(x,\tau)|^2dxd\tau&\leq C, \\
		\label{bnd-wN-L2}\|\omega^N\|_{L^\infty(0,T;L^2(\Omega))}&\leq C, \\
		\label{bnd-M-uN}\|M_\theta(u^N)\|_{L^\infty(0,T;L^\infty(\Omega))} &\leq C, \\
		\|W'(u^N)\|_{L^\infty(0,T;L^\infty(\Omega))} &\leq C, \\
		\label{bnd-W''-Linf}\|W''(u^N)\|_{L^\infty(0,T;L^\infty(\Omega))} &\leq C, \\
		\label{bnd-W'''-Linf}\|W'''(u^N)\|_{L^\infty(0,T;L^\infty(\Omega))} &\leq C, \\
		\label{bnd-uN_t}\|\partial_tu^N\|_{L^2(0,T;(H^2(\Omega))')}&\leq C, \\
		\label{bnd-uN-H5}\|u^N\|_{L^2(0,T;H^5(\Omega))}&\leq C_\theta.
	\end{align}
\end{lemma}

The estimates in Lemma \ref{bnd-lem} give the uniform bound for $c^N_1,...,c^N_N$. 
Therefore a global solution for the initial value problem \qref{ode4}-\qref{ode6} exists. 
With the specific form \qref{M-theta} of the positive mobility $M_\theta(u)$, we obtain the following bound for $\nabla\mu^N$.
\begin{lemma}
	Let $u^N$ be a solution of the system \qref{ode1}-\qref{ode3}, we have
	\begin{align}\label{bnd-grad-muN}
	\|\nabla\mu^N\|_{L^2(\Omega_T)}\leq \frac{C}{\theta^{1/2}}.
	\end{align}
\end{lemma}

\begin{proof}
By the definition of $M_\theta(u)$ in \qref{M-theta}, $M_\theta(u)\geq\theta$. Combining with \qref{bnd-M-muN} we 
obtain the estimate \qref{bnd-grad-muN}. 
\end{proof}

By the Aubin-Lions lemma and \qref{bnd-uN-H5} and \qref{bnd-uN_t}, there exist a subsequence of 
$\{u^N\}$ (not relabeled) and a function 
$u_\theta\in L^\infty(0,T;H^5(\Omega))\cap C([0,T];H^l(\Omega))\cap C([0,T];C^{3,\alpha}(\bar{\Omega}))$ 
such that as $N\rightarrow\infty$,
\begin{align}
	u^N&\rightharpoonup u_\theta \quad \mbox{weakly-* in}\; L^\infty(0,T;H^5(\Omega)), \\
	u^N&\rightarrow u_\theta \quad \mbox{strongly\; in}\; C([0,T];H^l(\Omega)), \\
	\label{u-conv2-C}u^N&\rightarrow u_\theta \quad \mbox{strongly\; in}\; C([0,T];C^{3,\alpha}(\bar{\Omega})) \; 
	\mbox{and \; a.e. \; in} \; \Omega_T, \\
	\partial_tu^N&\rightharpoonup \partial_tu_\theta \quad \mbox{weakly\; in}\; L^2(0,T;(H^2(\Omega))'),
\end{align}
where $1\leq l \leq 4$ and $0<\alpha<1/2$. 
Then from the Subsection 3.2 in \cite{dlp:weakFCH}, $u_\theta$ is a weak solution to the the 
FCH equation \qref{fch-eqn1}-\qref{fch-eqn3} in the following weak sense:
\begin{align}
	&\int_0^T\left\langle \partial_tu_\theta,\xi\right\rangle_{(H^2(\Omega)',H^2(\Omega))}dt 
	=-\int_0^T\int_\Omega M_\theta(u_\theta)\biggl(-\nabla\Delta\omega_\theta \nn\\
	&\quad\quad+W'''(u_\theta)
	\nabla u_\theta\omega_\theta+W''(u_\theta)\nabla\omega_\theta-\eta\nabla\omega_\theta\biggr)\cdot\nabla\xi dxdt
	\nn
\end{align}
for all $\xi\in L^2(0,T;H^2(\Omega))$, where $\omega_\theta=-\Delta u_\theta+W'(u_\theta)$.

\subsection{Energy inequality}
Fix any $t\geq 0$. 
Since
\begin{align*}
	\frac{d}{d\tau}\cF(u^N(x,\tau))=-\int_\Omega M_\theta(u^N(x,\tau))|\nabla\mu^N(x,\tau)|^2dx,
\end{align*}
integrating in time over $(0,t)$ gives the following energy identity
\begin{align}\label{ener-id}
	&\int_\Omega\frac{1}{2}|\omega^N(x,t)|^2-\eta\left(\frac{1}{2}|\nabla u^N(x,t)|^2+W(u^N(x,t))\right)dx \nonumber\\
	&+\int_{0}^{t}\int_\Omega M_\theta(u^N(x,\tau))|\nabla\mu^N(x,\tau)|^2dxd\tau  \nonumber\\
	&=\int_\Omega\frac{1}{2}|\omega^N(x,0)|^2-\eta\left(\frac{1}{2}|\nabla u^N(x,0)|^2+W(u^N(x,0))\right)dx.
\end{align}

Since $H^1(\Omega)\subset\subset L^2(\Omega)\hookrightarrow (H^3(\Omega))'$, by Aubin-Lions lemma,
\begin{align*}
	\{f\in 	L^\infty(0,T;H^1(\Omega)):\partial_tf\in L^2(0,T;(H^3(\Omega))')\}\subset\subset C(0,T;L^2(\Omega)).
\end{align*}
By \qref{bnd-uN-H2} and \qref{bnd-uN_t}, we have 
\[ \nabla u^N\in L^\infty(0,T;H^1(\Omega)) \quad\mbox{ and }\quad\partial_t\nabla u^N\in L^2(0,T;(H^3(\Omega))').\] 
Then by \qref{bnd-uN-H2} again, there exists a subsequence of $\{u^N\}$ (not relabeled) such that 
$\nabla u^N\rightarrow\nabla u_\theta$ strongly in $C([0,T];L^2(\Omega))$. Hence
\begin{align}\label{grad-uN(t)-conv}
	\nabla u^N(\cdot,t)\rightarrow\nabla u_\theta(\cdot,t)\; \mbox{strongly\; in}\;L^2(\Omega).
\end{align}
By \qref{u-conv2-C} and \qref{grow-1}, we have $W(u^N) \rightarrow W(u_\theta)$ in $C([0,T];L^2(\Omega))$, hence
\begin{align}\label{W-uN(t)-conv}
	W(u^N(\cdot,t)) \rightarrow W(u_\theta(\cdot,t)) \quad \mbox{strongly \; in} \; L^1(\Omega).
\end{align}
Since $H^1(\Omega)\subset\subset L^2(\Omega)\hookrightarrow (H^4(\Omega))'$, by Aubin-Lions lemma,
\begin{align*}
	\{f\in L^2(0,T;H^1(\Omega)):\partial_tf\in L^2(0,T;(H^4(\Omega))')\}\subset\subset L^2(0,T;L^2(\Omega)).
\end{align*}
By \qref{bnd-uN-H5} and \qref{bnd-uN_t}, we have 
\[ \Delta u^N\in L^2(0,T;H^1(\Omega)) \quad\mbox{ and }\quad\partial_t\Delta u^N\in L^2(0,T;(H^4(\Omega))').\]
Then by \qref{bnd-uN-H5} again, there exists a further subsequence of $\{u^N\}$ (not relabeled) such that 
$\Delta u^N\rightarrow\Delta u_\theta$ strongly in $L^2(0,T;L^2(\Omega))$. So
\begin{align}\label{conv-lapl-uN(t)}
	\Delta u^N(x,t)\rightarrow\Delta u_\theta(x,t)\; \mbox{for\;a.e.}\;x\in\Omega.
\end{align}	
By \qref{u-conv2-C} and \qref{grow-2}, we have $W'(u^N(\cdot,t))\rightarrow W'(u_\theta(\cdot,t))$ strongly in 
$L^q(\Omega)$ ($1\leq q<\infty$), so there exists a further subsequence of $\{u^N\}$ (not relabeled) such that
\begin{align}\label{W'-uN(t)-conv}
	W'(u^N(x,t))\rightarrow W'(u_\theta(x,t))\;\mbox{for\;a.e.}\;x\in\Omega.
\end{align}	
So 
\[\omega^N(x,t)=-\Delta u^N(x,t)+W'(u^N(x,t))\rightarrow\omega_\theta(x,t)=-\Delta u_\theta(x,t)+W'(u_\theta(x,t)) \]
 for a.e.  $x\in\Omega$. 
By Fatou's lemma,
\begin{align}\label{Fatou}
	\int_\Omega|\omega_\theta(x,t)|^2dx\leq\liminf_{N\rightarrow\infty}\int_\Omega|\omega^N(x,t)|^2dx.
\end{align}
From \cite{dlp:weakFCH} we have  $\sqrt{M_\theta(u^N)}\nabla\mu^N\rightharpoonup\chi_\theta
=\sqrt{M_\theta(u_\theta)}\nabla\mu_\theta$ weakly in $L^2(0,t;L^2(\Omega))$, then
\begin{align}\label{ineq-M-gradmu}
	&\int_0^{t}\int_\Omega M_\theta(u_\theta(x,\tau))|\nabla\mu_\theta(x,\tau)|^2dxd\tau \nn\\
	\leq & \liminf_{N\rightarrow\infty}
	\int_0^{t}\int_\Omega M_\theta(u^N(x,\tau))|\nabla\mu^N(x,\tau)|^2dxd\tau.
\end{align}
Since \qref{grad-uN(t)-conv} and \qref{W-uN(t)-conv} are true for arbitrary $t\geq 0$, we let $t=0$ in \qref{grad-uN(t)-conv} 
and \qref{W-uN(t)-conv} to get
\begin{align}
	\label{grad-uN(t0)-conv} \nabla u^N(\cdot,0)&\rightarrow\nabla u_0\; \mbox{strongly\;in}\;L^2(\Omega),  \\
	\label{W-uN(0)-conv} W(u^N(\cdot,0))&\rightarrow W(u_0)\; \mbox{strongly\;in}\;L^1(\Omega).
\end{align}
By \qref{u-conv2-C} and \qref{grow-2}, we have $W'(u^N)\rightarrow W'(u_\theta)$ strongly in $C([0,T];L^2(\Omega))$, so
\begin{align}\label{W'-uN(0)-conv}
	W'(u^N(\cdot,0))\rightarrow W'(u_\theta(\cdot,0))=W'(u_0)\;\mbox{strongly\;in}\;L^2(\Omega).
\end{align}
By \qref{ode3}, $u^N(\cdot,0)\rightarrow u_0$ strongly in $H^2(\Omega)$, then
\begin{align}\label{lapl-uN(0)-conv}
	\Delta u^N(\cdot,0)\rightarrow\Delta u_0\;\mbox{strongly\;in}\;L^2(\Omega).
\end{align}
Combining \qref{W'-uN(0)-conv} and \qref{lapl-uN(0)-conv} we have 
\[\omega^N(\cdot,0)=-\Delta u^N(\cdot,0)+W'(u^N(\cdot,0))\rightarrow -\Delta u_0+W'(u_0)\] 
strongly in $L^2(\Omega)$, so
\begin{align}\label{int-w(0)}
	\lim_{N\rightarrow\infty}\int_\Omega|\omega^N(x,0)|^2dx=\int_\Omega|-\Delta u_0+W'(u_0)|^2dx.
\end{align}
So we can find a subsequence of $\{u^N\}$ (not relabeled) that satisfies \qref{ener-id}-\qref{int-w(0)}. Then 
by taking the limit as $N\rightarrow\infty$ in \qref{ener-id}, we get the energy inequality \qref{ener-ineq1}.

\subsection{Positive initial data}
In this subsection, we assume that the initial data $u_0(x)>0$ for all $x\in\Omega$. 
Let $\Phi:(0,\infty)\rightarrow [0,\infty)$ be defined by
\begin{align*}
	\Phi''(u)=\frac{1}{M(u)}, \quad \Phi(1)=\Phi'(1)=0,
\end{align*}
and $\Phi_\theta$ be defined by
\begin{align*}
	\Phi_\theta''(u)=\frac{1}{M_\theta(u)}, \quad \Phi_\theta(1)=\Phi_\theta'(1)=0, \quad \mbox{and} \quad \Phi_\theta\in C^2(\R),
\end{align*}
for $0<\theta<1$. Using the definition of $M(u)$ and $M_\theta(u)$ in \qref{M(u)} and \qref{M-theta}, by direct 
calculation, we have
\begin{align}\label{Phi(u)}
	\Phi(u)=u\ln u-u+1 \quad\mbox{for}\; u>0,
\end{align}
and
\begin{align}\label{Phi-theta}
	\Phi_\theta(u)=\left\{\begin{array}{l} u\ln u-u+1, \qquad u>\theta,  \\
	\frac{1}{2\theta}u^2+(\ln\theta-1)u+1-\frac{\theta}{2}, \; u\leq\theta. \end{array} \right.
\end{align}
We notice that  $\Phi_\theta(u)\geq 0$ for all $u\in\R$, $0\leq\Phi_\theta(u)\leq\Phi(u)$ for all $u>0$ and $\Phi_\theta(u)=\Phi(u)$ for all $u\geq\theta$.

\textbf{Claim 1.} \textit{For any $t\in [0,T]$,}
\begin{align}\label{Phi-theta-eqn}
	\int_\Omega\Phi_\theta(u_\theta(x,t))dx - \int_\Omega\Phi_\theta(u_0(x))dx=-\int_0^t\int_\Omega \nabla\mu_\theta\cdot\nabla u_\theta dxd\tau.
\end{align}

To prove this, for any $\epsilon>0$, let $\Phi_{\theta,\epsilon}$ be the mollification of $\Phi_\theta$. 
Since $\Phi_\theta \in C^2(\R)$, then $\Phi_{\theta,\epsilon} \to \Phi_\theta, \Phi'_{\theta,\epsilon} \to \Phi'_\theta$ and $\Phi''_{\theta,\epsilon} \to \Phi''_\theta$ on compact subsets of $\R$ as $\epsilon \to 0$. 
First, we prove that
\begin{align}\label{Phi-theta-eps-eqn}
	&\int_\Omega\Phi_{\theta,\epsilon}(u_\theta(x,t))dx - \int_\Omega\Phi_{\theta,\epsilon}(u_0(x))dx\nn\\
	=&-\int_0^t\int_\Omega M_\theta(u_\theta)\Phi_{\theta,\epsilon}''(u_\theta)\nabla\mu_\theta\cdot\nabla u_\theta dxd\tau.
\end{align}
For any $h>0$, define 
\begin{align}
	u_{\theta,h}(x,t):=\frac{1}{h}\int_{t-h}^ tu_\theta(x,\tau)d\tau,
\end{align}
where we set $u_\theta(x,t)=u_0(x)$ when $t\leq 0$. 
Since $H^3(\Omega)\subset\subset H^2(\Omega)\hookrightarrow (H^3(\Omega))'$, by the Aubin-Lions lemma,
\begin{align*}
	\{f\in 	L^2(0,T;H^3(\Omega)):\partial_tf\in L^2(0,T;(H^3(\Omega))')\}\subset\subset L^2(0,T;H^2(\Omega)).
\end{align*}
Since $u_\theta\in L^\infty(0,T;H^5(\Omega))$ and $\Phi_{\theta,\epsilon}^{(k)}(k=1,2,3,4)$ are bounded, we have
\begin{align*}
	\sup_{h>0}\|\Phi_{\theta,\epsilon}'(u_{\theta,h})\|_{L^2(0,T;H^3(\Omega))} \leq C_\theta \quad \mbox{and} 
	\quad  \partial_t\Phi_{\theta,\epsilon}'(u_{\theta,h})\in L^2(0,T;(H^3(\Omega))')
\end{align*}
for any $h>0$. Hence there exists a subsequence of $\{\Phi_{\theta,\epsilon}'(u_{\theta,h})\}_{h>0}$ (not relabeled) such that
\begin{align} \label{conv-Phi'-u-theta-h}
	\Phi_{\theta,\epsilon}'(u_{\theta,h}) \rightarrow \Phi_{\theta,\epsilon}'(u_\theta) \quad \mbox{strongly \; in} \; L^2(0,T;H^2(\Omega)) \; 
	\mbox{as} \; h\rightarrow 0.
\end{align}
Furthermore, we can show that
\begin{align} \label{conv-dt/du-theta-h}
	\partial_tu_{\theta,h} \rightarrow \partial_tu_\theta \quad \mbox{strongly \; in} \; 
	L^2(0,T;(H^2(\Omega))') \; \mbox{as} \; h\rightarrow 0.
\end{align}
Indeed, let $\textup{\textbf{J}}_\theta:=M_\theta(u_\theta)\nabla\mu_\theta$, then for any $\xi\in L^2(0,T;H^2(\Omega))$,
\begin{align*}
	&\left| \left\langle \partial_tu_{\theta,h} - \partial_tu_\theta,\xi\right\rangle_{(L^2(0,T;(H^2(\Omega))'),
	L^2(0,T;H^2(\Omega))} \right| \nn\\
	=& \frac{1}{h}\left| \int_0^T\left\langle \int_{t-h}^t(\partial_tu_\theta(\tau) - \partial_tu_\theta(t))d\tau,
	\xi\right\rangle_{((H^2(\Omega))',H^2(\Omega))} dt \right| \\ 
	=& \frac{1}{h}\left| \int_0^T\left\langle \int_{-h}^0(\partial_tu_\theta(t+s) - \partial_tu_\theta(t))ds,
	\xi\right\rangle_{((H^2(\Omega))',H^2(\Omega))} dt \right| \\  
	\leq& \frac{1}{h} \int_{-h}^0 \left| \int_0^T \int_\Omega \nabla\xi\cdot (\textup{\textbf{J}}_\theta(t+s)
	-\textup{\textbf{J}}_\theta(t))dxdt\right| ds \\
	\leq& \|\xi\|_{L^2(0,T;H^2(\Omega))}\sup_{-h\leq s \leq 0}\|\textup{\textbf{J}}_\theta(\cdot+s)
	-\textup{\textbf{J}}_\theta(\cdot)\|_{L^2(0,T;H^2(\Omega))}.
\end{align*}
Hence,
\begin{align}
	\|\partial_tu_{\theta,h} - \partial_tu_\theta\|_{L^2(0,T;(H^2(\Omega))')} \leq \|\textup{\textbf{J}}_\theta(\cdot+s)
	-\textup{\textbf{J}}_\theta(\cdot)\|_{L^2(0,T;H^2(\Omega))} \rightarrow 0 	\nn
\end{align}
as $h\to 0$. 
So we obtain \qref{conv-dt/du-theta-h}.

Since $\Phi_{\theta,\epsilon}'(u_{\theta,h})$ and $\partial_tu_{\theta,h}$ are both in $L^2(\Omega_T)$, we 
have for a.e. $t\in[0,T]$,
\begin{align}
	&\int_0^t\left\langle \partial_tu_{\theta,h},\Phi_{\theta,\epsilon}'(u_{\theta,h})
	\right\rangle_{((H^2(\Omega))',H^2(\Omega))}d\tau \nn\\
	=& \int_0^t \int_\Omega \Phi_{\theta,\epsilon}'(u_{\theta,h})\partial_tu_{\theta,h} dxd\tau \nonumber\\
	=& \int_\Omega \int_0^t \partial_t \Phi_{\theta,\epsilon}(u_{\theta,h}(\tau,x)) d\tau dx  \nonumber\\
	=&\int_\Omega\Phi_{\theta,\epsilon}(u_{\theta,h}(x,t))dx - \int_\Omega \Phi_{\theta,\epsilon}(u_0(x))dx. \nn
\end{align}
Passing to the limit as $h\rightarrow 0$ in this equation, combining with \qref{conv-Phi'-u-theta-h} and 
\qref{conv-dt/du-theta-h}, we get
\begin{align}\label{Phi-theta-eps-lhs}
	&\int_0^t\left\langle \partial_tu_\theta,\Phi_{\theta,\epsilon}'(u_\theta)\right\rangle_{((H^2(\Omega))',H^2(\Omega))}d\tau 
	\nn\\
	=& \int_\Omega \Phi_{\theta,\epsilon}(u_\theta(x,t))dx - \int_\Omega\Phi_{\theta,\epsilon}(u_0(x))dx.
\end{align}
Since $\Phi_{\theta,\epsilon}''$ and $\Phi_{\theta,\epsilon}'''$ are bounded, $\Phi_{\theta,\epsilon}'(u_\theta)\in L^2(0,T;H^2(\Omega))$. 
So $\Phi_{\theta,\epsilon}'(u_\theta)$ is an admissible test function for the equation \qref{FCH-w2}. Hence for 
any $t\in [0,T]$,
\begin{align}\label{Phi-theta-eps-rhs}
	&\int_0^t\left\langle \partial_tu_\theta,\Phi_{\theta,\epsilon}'(u_\theta)\right\rangle_{((H^2(\Omega))',H^2(\Omega))}d\tau 
	\nn\\
	=&-\int_0^t\int_\Omega 		M_\theta(u_\theta)\nabla\mu_\theta\cdot\nabla(\Phi_{\theta,\epsilon}'(u_\theta)) dxd\tau 
	\nonumber\\
	=&-\int_0^t\int_\Omega M_\theta(u_\theta)\Phi_{\theta,\epsilon}''(u_\theta)\nabla\mu_\theta\cdot\nabla u_\theta dxd\tau.
\end{align}
Combining \qref{Phi-theta-eps-lhs} and \qref{Phi-theta-eps-rhs} we get \qref{Phi-theta-eps-eqn}.

For each $t \in [0,T]$, since $u_\theta \in C([0,T];C^{3,\alpha}(\bar{\Omega}))$, $u_\theta(\bar{\Omega},t):=\{u_\theta(x,t): x\in\bar{\Omega}\}$ is a compact subset of $\R$, 
so $\Phi_{\theta,\epsilon} \to \Phi_\theta$ and $\Phi''_{\theta,\epsilon} \to \Phi''_\theta$ uniformly on $u_\theta(\bar{\Omega},t)$ as $\epsilon \to 0$, 
hence $\Phi_{\theta,\epsilon}(u_\theta) \to \Phi_\theta(u_\theta)$ and $\Phi''_{\theta,\epsilon}(u_\theta) \to \Phi''_\theta(u_\theta)$ uniformly on $\bar{\Omega}$ as $\epsilon \to 0$. 
Also, $u_\theta(\overline{\Omega_T}):=\{u_\theta(x,\tau): x\in\bar{\Omega},\tau\in[0,T]\}$ is a compact subset of $\R$, 
so $\Phi''_{\theta,\epsilon} \to \Phi''_\theta$ uniformly on $u_\theta(\overline{\Omega_T})$ as $\epsilon \to 0$, 
hence $\Phi''_{\theta,\epsilon}(u_\theta) \to \Phi''_\theta(u_\theta)$ uniformly on $\overline{\Omega_T}$ as $\epsilon \to 0$. 
Using Lemma \ref{bnd-lem}, we have
\begin{align} \nn
	\|M_\theta(u_\theta)\|_{L^\infty(\Omega_T)} <\infty \quad \mathrm{and} \quad 
	\|\sqrt{M_\theta(u_\theta)}\nabla\mu_\theta\|_{L^2(\Omega_T)}<\infty.
\end{align}
Since $u_\theta \in C([0,T];C^{3,\alpha}(\bar{\Omega}))$, we have $\|\nabla u_\theta\|_{L^\infty(\Omega_T)} <\infty$. 
Thus
\begin{align}
	&\left| \int_0^t\int_\Omega M_\theta(u_\theta)\nabla\mu_\theta\cdot\nabla u_\theta
	\left( \Phi_{\theta,\epsilon}''(u_\theta)-\Phi_\theta''(u_\theta)\right)  dxd\tau\right|  \nonumber \\
	\leq &\|M_\theta(u_\theta)\|_{L^\infty(\Omega_T)}^{1/2}\|\nabla u_\theta\|_{L^\infty(\Omega_T)} \|\sqrt{M_\theta(u_\theta)}\nabla\mu_\theta\|_{L^2(\Omega_T)}\nn\\
	&\cdot \|\Phi_{\theta,\epsilon}''(u_\theta)-\Phi_\theta''(u_\theta)\|_{L^2(\Omega_T)}  \nonumber\\
	\to &0 \quad \mathrm{as} \; \epsilon \to 0. \nn
\end{align}
So passing to limits as $\epsilon \to 0$ on both sides of \qref{Phi-theta-eps-eqn} and using $\Phi''_\theta(u_\theta)=1/M_\theta(u_\theta)$, we get \qref{Phi-theta-eqn}.

\textbf{Claim 2.} \textit{Let $(u_\theta)_-:=\min\{u_\theta,0\}$, then for any $0<\theta<1$,}
\begin{align}\label{bnd-neg2}
	\textup{ess sup}_{0\leq t \leq T}\int_\Omega |(u_\theta)_-+\theta|^2dx\leq C(\theta^2+\theta+\theta^{1/2}).
\end{align}
To prove this, for any $z\leq0$, we rewrite $\Phi_\theta(z)$ as
\begin{align}\label{phi-theta-neg}
	\Phi_\theta(z) =\frac{1}{2\theta}(z+\theta)^2+(\ln\theta-2)z+(1-\theta).
\end{align}
Since $0<\theta<1$, \qref{phi-theta-neg} implies
\begin{align}
	(z+\theta)^2 \leq 2\theta \Phi_\theta(z) \quad \mathrm{for\; all} \;z\leq 0. \nn
\end{align}
Hence for any $t\in[0,T]$, since $\Phi_\theta(z) \geq 0$ for all $z\in\R$, we have
\begin{align}\label{bnd-u-theta-neg}
	&\int_\Omega |(u_\theta(x,t))_-+\theta|^2dx \nn\\
	\leq& 2 \theta\int_\Omega\Phi_\theta(u_\theta(x,t)_-)dx  \nonumber \\
	\leq & 2\theta \left( \int_{\{u_\theta\leq 0\}}\Phi_\theta(u_\theta(x,t))dx+\int_{\{u_\theta>0\}}\Phi_\theta(0)dx\right) 
	\nonumber \\
	\leq& 2\theta \left( \int_\Omega\Phi_\theta(u_\theta(x,t))dx+\int_\Omega\Phi_\theta(0)dx\right)  \nonumber\\
	\leq& 2\theta \left[ \int_\Omega\Phi_\theta(u_\theta(x,t))dx+\left(1-\frac{\theta}{2} \right)|\Omega|\right]. 
\end{align} 
From \qref{bnd-grad-muN} and \qref{bnd-uN-H2}, we have $\|\nabla\mu_\theta\|_{L^2(\Omega_T)} \leq C/\theta^{1/2}$ 
and $\|\nabla u_\theta\|_{L^2(\Omega_T)} \leq C$. Thus by \qref{Phi-theta-eqn} and H\"{o}lder's inequality, we have
\begin{align}\
	&\left|\int_\Omega\Phi_\theta(u_\theta(x,t))dx\right| \nn\\
	\leq& \left|\int_\Omega\Phi_\theta(u_0)dx\right| 
	+ \left|\int_0^t\int_\Omega \nabla\mu_\theta(x,\tau)\cdot\nabla u_\theta(x,\tau) dxd\tau\right|  \nonumber\\
	\leq & \int_\Omega\Phi_\theta(u_0)dx + \|\nabla\mu_\theta\|_{L^2(\Omega_T)}\|\nabla u_\theta\|_{L^2(\Omega_T)} 	
	\nonumber \\
	\leq & \int_\Omega\Phi(u_0)dx + \frac{C}{\theta^{1/2}} \nn
\end{align}
for any $t\in[0,T]$. 
This inequality and \qref{bnd-u-theta-neg} implies \qref{bnd-neg2}. $\square$

\section{Weak solutions for the degenerate mobility case}\label{dege-mobi}
In this section we prove the main theorem that is Theorem \ref{thm-main2}.  
We now consider the FCH equation \qref{fch-eqn1}-\qref{fch-eqn3} with the degenerate mobility $M(u)$ defined by \qref{M(u)}. 
The proof for the existence of a weak solution $u$ and a function $\zeta$ that satisfy \qref{FCH-w3} is similar to that in 
Section 4.4 in \cite{dlp:weakFCH}. 
Again, here we just sketch the idea of that proof, as well as state the convergences and estimates that are necessary for 
later parts.  
In this section we provide a more detailed proof for the relation between $u$ and $\zeta$, as well as prove the energy in 
equality \qref{ener-ineq2}. 
Moreover, we prove the existence of a nonnegative solution that is not constantly zero in $\Omega_T$ when the initial 
data $u_0$ is positive in $\Omega$.

\subsection{Weak solution to the degenerate FCH equation}
Fix $u_0\in H^2(\Omega)$ and a sequence $\{\theta_i\}_{i=1}^\infty \subset (0,1)$ that monotonically 
decreases to 0 as $i\rightarrow\infty$. 
For each $\theta_i$, by Theorem \ref{thm-main1}, there exists a function
\begin{align*}
	u_{\theta_i}\in L^\infty(0,T;H^5(\Omega))\cap C([0,T];H^l(\Omega))\cap C([0,T];C^{3,\alpha}(\bar{\Omega})), 
\end{align*}
whose weak derivative is
\begin{align*}
	\partial_tu_{\theta_i}\in L^2(0,T;(H^2(\Omega))'),
\end{align*}
where $1\leq l \leq 4$ and $0<\alpha<1/2$, such that for any $\xi\in L^2(0,T;H^2(\Omega))$,
\begin{align}
	\int_0^T\left\langle \partial_tu_{\theta_i},\xi\right\rangle_{(H^2(\Omega)',H^2(\Omega))}dt	
	&=-\int_0^T\int_\Omega M_{\theta_i}(u_{\theta_i})\nabla\mu_{\theta_i}\cdot\nabla\xi dxdt, \nn\\
	\mu_{\theta_i}&=-\Delta\omega_{\theta_i}+W''(u_{\theta_i})\omega_{\theta_i}-\eta\omega_{\theta_i}, \nn\\
	\omega_{\theta_i}&=-\Delta u_{\theta_i}+W'(u_{\theta_i}), \nn\\
	u_{\theta_i}(x,0)&=u_0(x) \quad \mbox{for \; all}\; x\in\Omega. \nn
\end{align}
For simplicity, we write $u_i:=u_{\theta_i},\mu_i:=\mu_{\theta_i},\omega_i:=\omega_{\theta_i}$, and $M_i:=M_{\theta_i}$. 
In Lemma \ref{bnd-lem}, the bounds on the right hand side of \qref{bnd-uN-H2}, \qref{bnd-wN-L2} 
and \qref{bnd-uN_t} depend only on $d,T,\Omega,\eta,u_0$ and $C_j(j=1,2,3,4)$ but not on $\theta$. 
So there exists a constant $C>0$ independent on $\{\theta_i\}_{i=1}^\infty$ such that
\begin{align}
	\|u_i\|_{L^\infty(0,T;H^2(\Omega))} &\leq C, \\
	\|\partial_tu_i\|_{L^2(0,T;(H^2(\Omega))')} &\leq C, \\
	\label{bnd-wi}\|\omega_i\|_{L^\infty(0,T;L^2(\Omega))} &\leq C.
\end{align}
Hence by the Aubin-Lions lemma, there exist a subsequence of $\{u_i\}$ (not relabeled) and a function 
$u\in L^\infty(0,T;H^2(\Omega))\cap C([0,T];H^1(\Omega))\cap C([0,T];C^\alpha(\bar{\Omega}))$ such that as 
$i\rightarrow\infty$,
\begin{align}
	u_i&\rightharpoonup u \quad \mbox{weakly-*\; in}\; L^\infty(0,T;H^2(\Omega)), \\
	\label{ui-conv-C}u_i&\rightarrow u \quad \mbox{strongly\; in}\; C([0,T];C^\alpha(\bar{\Omega})), \\
	\label{ui-conv-W2p}u_i&\rightarrow u \quad \mbox{strongly\; in}\; C([0,T];H^1(\Omega)) \; 
	\mbox{and \; a.e. \; in} \; \Omega_T, \\
	\partial_tu_i&\rightharpoonup \partial_tu \quad \mbox{weakly\; in}\; L^2(0,T;(H^2(\Omega))'),
\end{align}
where $0<\alpha<1/2$. 

Choose a sequence of positive numbers $\{\delta_j\}_{j=0}^\infty$ that monotonically decreases to $0$. 
For each $\delta_j$, by \qref{ui-conv-W2p} and Egorov's theorem, there exists a subset $B_j\subset\Omega_T$ 
with $|\Omega_T\backslash B_j|<\delta_j$ such that $u_i\rightarrow u$ in $B_j$ as $i\rightarrow\infty$. 
We may take	$B_1\subset B_2\subset ...\subset B_j\subset B_{j+1}\subset ...\subset\Omega_T$. 
Define 
\begin{align*}
	B:=\bigcup_{j=1}^{\infty} B_j,
\end{align*}
then $|\Omega_T\backslash B|=0$. 
We also define
\begin{align*}
	P_j:=\{(x,t)\in\Omega_T: u(x,t)>\delta_j\},
\end{align*}
then $P_1\subset P_2\subset ...\subset P_j\subset P_{j+1}\subset ...\subset\Omega_T$. 
Let
\begin{align*}
	P:=\bigcup_{j=1}^{\infty} P_j = \{(x,t)\in\Omega_T:u(x,t)>0\},
\end{align*}
that is, $P$ is the set where $M(u)$ is not degenerate. 
Using a diagonal argument as in Subsection 4.1 in \cite{dlp:weakFCH}, we are able to find a subsequence 
$\{u_{k,N_k}\}_{k=1}^\infty$, which converges to $u$, 
and a function $\zeta:\Omega_T\rightarrow\R^d$ satisfying $\chi_{B\cap P}M(u)\zeta\in L^2(0,T;L^{2d/(d+2)}(\Omega))$ 
such that $u$ and $\zeta$ solve the following weak equation
\begin{align}\label{FCH-w4}
	\int_0^T\left\langle\partial_tu,\phi\right\rangle_{(H^2(\Omega)',H^2(\Omega))}dt
	=-\int_{B\cap P}M(u)\zeta\cdot\nabla\phi dxdt
\end{align}
for all $\phi\in L^2(0,T;H^2(\Omega))$.

\subsection{The relation between $\zeta$ and $u$}
Since $u\in L^\infty(0,T;H^2(\Omega))$, the function $\omega:=-\Delta u+W'(u)$ is completely defined, 
and $\omega\in L^\infty(0,T;L^2(\Omega))$.
The desired relation between $\zeta$ and $u$ is
\begin{align*}
	\zeta=-\nabla\Delta\omega+W'''(u)\nabla u\omega+W''(u)\nabla\omega-\eta\nabla\omega.
\end{align*}
But the terms $\nabla\omega$ and $\nabla\Delta\omega$ are only defined in the sense of distributions 
and may not even be functions. 
So we need some higher regularity conditions on $u$.

\textbf{Claim.} \textit{For any open set $U\subset\Omega$ such that $\nabla\Delta^2u\in L^q(U_T)$ for 
some $q>1$ that may depend on $U$, where $U_T=U\times(0,T)$, we have}
\begin{align} \nn
	\zeta=-\nabla\Delta\omega+W'''(u)\nabla u\omega+W''(u)\nabla\omega-\eta\nabla\omega \quad \mbox{in}\; U_T.
\end{align}

By the equation (71) in \cite{dlp:weakFCH}, we have 
\begin{align}\label{grad-muk}
	\nabla\mu_{k,N_k}=&-\nabla\Delta\omega_{k,N_k}+W'''(u_{k,N_k})\nabla u_{k,N_k}
	\omega_{k,N_k} \nn\\
	&+W''(u_{k,N_k})\nabla\omega_{k,N_k}-\eta\nabla\omega_{k,N_k}.	
\end{align}
Since $\nabla\Delta^2u\in L^q(U_T)$ and 
\[u\in L^\infty(0,T;H^2(\Omega))\cap C([0,T];H^1(\Omega))\cap C([0,T];C^\alpha(\bar{\Omega})),\] 
where $q>1$ and $0<\alpha<1/2$, we have
\begin{align*}
	\Delta^2u &\in L^q(0,T;W^{1,q}(U)), \\	
	\nabla\Delta u &\in L^q(0,T;W^{2,q}(U)), \\
	\Delta u &\in L^q(0,T;W^{3,q}(U))\cap L^\infty(0,T;L^2(U)), \\
	\nabla u &\in L^q(0,T;W^{4,q}(U))\cap L^\infty(0,T;H^1(U))\cap C([0,T];L^2(U)), \\
	u &\in L^q(0,T;W^{5,q}(U))\cap L^\infty(0,T;H^2(U))\nn\\
	&\quad \cap C([0,T];H^1(U))\cap C([0,T];C^\alpha(\bar{U})),
\end{align*}
for the values of $q$ and $\alpha$ indicated.
By \qref{bnd-wi}, there exists a subsequence of $\{\omega_{k,N_k}\}_{k=1}^\infty$ (not relabeled) such that
\begin{align}\label{wk-wconv-L2}
	\omega_{k,N_k}\rightharpoonup \omega \quad \mbox{weakly\; in}\; L^2(0,T;L^2(U)).
\end{align}
Hence, as $k\rightarrow\infty$,
\begin{align}
	\label{grad-wk-wconv}\nabla\omega_{k,N_k}&\rightharpoonup \nabla\omega \quad \mbox{weakly\; in}\; L^2(0,T;(H^1(U))'),\\
	\label{grad-lapl-wk-wconv}\nabla\Delta\omega_{k,N_k}&\rightharpoonup \nabla\Delta\omega \quad \mbox{weakly\; in}
	\; L^2(0,T;(H^3(U))').
\end{align}
By \qref{ui-conv-W2p}, we have
\begin{align}\label{grad-uk-conv-L2}
	\nabla u_{k,N_k}\rightarrow \nabla u \quad \mbox{strongly\; in}\; C([0,T];L^2(U)).
\end{align}
We see that the bounds on the right hand side of \qref{bnd-W''-Linf} and \qref{bnd-W'''-Linf} depend only on 
$d,T,\Omega,\eta,u_0$ and $C_j(j=1,2,3,4)$ but not on $\theta$. 
So there exists a constant $C>0$ independent on $\{\theta_i\}_{i=1}^\infty$ such that
\begin{align}
	\label{bnd-W''-ui}\|W''(u_i)\|_{L^\infty(0,T;L^\infty(\Omega))} &\leq C, \\
\label{bnd-W'''-ui}\|W'''(u_i)\|_{L^\infty(0,T;L^\infty(\Omega))} 	&\leq C.
\end{align}
By \qref{bnd-W'''-ui}, \qref{wk-wconv-L2} and \qref{grad-uk-conv-L2}, we have
\begin{align}\label{W'''-graduk-wk-wconv}
	W'''(u_{k,N_k})\nabla u_{k,N_k}\omega_{k,N_k}\rightharpoonup W'''(u)\nabla u\omega 
	\quad \mbox{weakly\; in}\; L^2(0,T;L^2(U)).
\end{align}
Thus by \qref{grad-muk}, \qref{grad-wk-wconv}, \qref{grad-lapl-wk-wconv}, \qref{bnd-W''-ui} and 
\qref{W'''-graduk-wk-wconv}, we obtain
\begin{align}\label{grad-muk-wconv2}
	\nabla\mu_{k,N_k}\rightharpoonup -\nabla\Delta\omega+W'''(u)\nabla u\omega+W''(u)\nabla\omega-\eta\nabla\omega 
\end{align}
weakly in $L^2(0,T;(H^3(U))')$. 
From Section 4.1 in \cite{dlp:weakFCH}, for every $j=1,2,...$, $\nabla\mu_{k,N_k} \rightharpoonup\zeta$ weakly 
in $L^2(B_j\cap P_j)$ as $k\rightarrow\infty$. 
So by the uniqueness of weak limits, we get
\begin{align*}
	\zeta=-\nabla\Delta\omega+W'''(u)\nabla u\omega+W''(u)\nabla\omega-\eta\nabla\omega \quad\mbox{in}\; B\cap P\cap U_T.
\end{align*}
We may extend the function $\zeta$ from $B\cap P\cap U_T$ to $U_T$ by defining it to be 
$-\nabla\Delta\omega+W'''(u)\nabla u\omega+W''(u)\nabla\omega-\eta\nabla\omega$ in $U_T\backslash(B\cap P)$. 
This completes the proof of the claim.

Similar to \cite{dlp:weakFCH}, we define the set
\begin{align*}
	\tilde{\Omega}_T:=\bigcup\biggl\{U_T=U\times(0,T): &\; U \; \mbox{is\; open\; in\;}\Omega,\; \nabla\Delta^2u\in 
	L^q(U_T) \nn\\
	 &\mbox{ for some } q>1 
	\;\mbox{that\; may\; depend\; on\;} U\biggr\}.
\end{align*}
Then $\tilde{\Omega}_T$ is open in $\Omega_T$ and
\begin{align*}
	\zeta=-\nabla\Delta\omega+W'''(u)\nabla u\omega+W''(u)\nabla\omega-\eta\nabla\omega \quad\mbox{in}\; \tilde{\Omega}_T.
\end{align*}
So $\zeta$ is defined in $(B\cap P)\cup\tilde{\Omega}_T$. To extend $\zeta$ to $\Omega_T$, notice that
\begin{align*}
	\Omega_T\backslash((B\cap P)\cup\tilde{\Omega}_T)\subset (\Omega_T\backslash B)\cup(\Omega_T\backslash P).
\end{align*}
Since $|\Omega_T\backslash B|=0$ and $M(u)=0$ in $\Omega_T\backslash P$, the value of $\zeta$ outside 
of $(B\cap P)\cup\tilde{\Omega}_T$ does not contribute to the integral on the right hand side of \qref{FCH-w4}, 
so we may just let $\zeta=0$ outside of $(B\cap P)\cup\tilde{\Omega}_T$.

\subsection{Energy inequality}
By the energy inequality \qref{ener-ineq1}, for any $k,j=1,2,...$, we have
\begin{align}\label{ener-ineq3}
	&\int_\Omega\frac{1}{2}|\omega_{k,N_k}(x,t)|^2-\eta\left(\frac{1}{2}|\nabla u_{k,N_k}(x,t)|^2+W(u_{k,N_k}(x,t))\right)dx  
	 \nonumber \\
	&+\int_{\Omega_t\cap B_j\cap P_j} M_{k,N_k}(u_{k,N_k}(x,\tau))|\nabla\mu_{k,N_k}(x,\tau)|^2dxd\tau  \nonumber \\
	&\leq\int_\Omega\frac{1}{2}|-\Delta u_0+W'(u_0)|^2-\eta\left(\frac{1}{2}|\nabla u_0|^2+W(u_0)\right)dx.
\end{align}
Using the convergence in \qref{ui-conv-C} and \qref{ui-conv-W2p} and the fact (from \cite{dlp:weakFCH}) that
\begin{align*}
	\chi_{B_j\cap P_j}\sqrt{M_{k,N_k}(u_{k,N_k})}\nabla\mu_{k,N_k}\rightharpoonup\chi_{B_j\cap P_j}\sqrt{M(u)}\zeta
\end{align*}
weakly in $L^2(0,T;L^{2d/(d+2)}(\Omega)) $ 
as $k\rightarrow\infty$.
Taking the limit as $k\rightarrow\infty$ first and then $j\rightarrow\infty$ in \qref{ener-ineq3}, we obtain the 
energy inequality \qref{ener-ineq1}.

\subsection{Nonnegative weak solution with positive initial data}
Assume, in addition, that the initial data $u_0(x)>0$ for all $x\in\Omega$. 
By \qref{bnd-neg1}, there exists a constant $C$ independent on $\{\theta_i\}_{i=1}^\infty \subset (0,1)$ such that for each $i=1,2,...$,
\begin{align}\label{bnd-neg3}
	\textup{ess sup}_{0\leq t \leq T}\int_\Omega|(u_i(x,t))_-+\theta_i|^2dx\leq C(\theta_i^2+\theta_i+\theta_i^{1/2}).
\end{align}
By the convergence in \qref{ui-conv-C} and \qref{ui-conv-W2p}, and since $u$ is continuous in $\Omega_T$, passing to the limit as $i\rightarrow\infty$ in \qref{bnd-neg3} yields $u(x,t)\geq 0$ for all $(x,t)\in\Omega_T$. 
Furthermore, since $u_0>0$ in $\Omega$, by the continuity of $u$ in $\Omega_T$, $u$ is not constantly zero in $\Omega_T$. 
This completes the proof of Theorem \ref{thm-main2}.

\section*{Acknowledgments}

This work was supported by the US National Science Foundation (NSF)
 under grants DMS-1802863 and DMS-1815746 (Shibin Dai), DMS-1813203 (Keith Promislow), and 
 the NSF 
of Guangdong Province, China under grant 2020A1515010554 (Qiang Liu).

\bibliographystyle{plain}
\bibliography{dai-bib}

\end{document}